\documentclass{article}

\usepackage{graphicx}
\usepackage{indentfirst}
\usepackage{amsmath,amsfonts,amsthm,amssymb}
\usepackage{mathrsfs}
\usepackage{amscd}
\usepackage{hyperref}

\def\co{\colon\thinspace}
\DeclareMathAlphabet{\mathsfsl}{OT1}{cmss}{m}{sl}
\newcommand{\tensor}[1]{\mathsfsl{#1}}
\newcommand{\spin}{\mathrm{Spin}^c}

\newtheorem{thm}{Theorem}[section]
\newtheorem{lem}[thm]{Lemma}
\newtheorem{cor}[thm]{Corollary}
\newtheorem{prop}[thm]{Proposition}

\theoremstyle{definition}
\newtheorem{defn}[thm]{Definition}
\newtheorem{notn}[thm]{Notation}
\newtheorem{rem}[thm]{Remark}

\newcommand{\Z}{\mathbb{Z}}

\begin{document}

\title{Correction terms, $\mathbb Z_2$--Thurston norm, and triangulations}

\author{{\Large Yi NI}\\{\normalsize Department of Mathematics, Caltech, MC 253-37}\\
{\normalsize 1200 E California Blvd, Pasadena, CA
91125}\\{\small\it Email\/:\quad\rm yini@caltech.edu}\\
{\ }\\
{\Large Zhongtao WU}\\
{\normalsize Department of Mathematics,
The Chinese University of Hong Kong,}\\
{\normalsize Shatin, Hong Kong}\\{\small\it Email\/:\quad\rm ztwu@math.cuhk.edu.hk}
}

\date{}
\maketitle

\begin{abstract}
We show that the correction terms in Heegaard Floer homology give a lower bound to the the genus of one-sided Heegaard splittings and the $\mathbb Z_2$--Thurston norm. Using a result of Jaco--Rubinstein--Tillmann, this gives a lower bound to the complexity of certain closed $3$--manifolds. As an application, we compute the $\mathbb Z_2$--Thurston norm of the double branched cover of some closed 3--braids, and give upper and lower bounds for the complexity of these manifolds.
\end{abstract}

\section{Introduction}

Heegaard Floer homology, introduced by Ozsv\'ath and Szab\'o \cite{OSzAnn1}, has been very successful in the study of low-dimensional topology. One important feature of Heegaard Floer homology which makes it so useful is that it gives a lower bound for the genus of surfaces in a given homology class. In dimension $3$, it determines the Thurston norm \cite{OSzGenus}. In dimension $4$, the adjunction inequality \cite{OSz4Manifold} gives a lower bound to the genus of surfaces which is often sharp, and the concordance invariant \cite{OSz4Genus} gives a lower bound to the slice genus of knots. 

In \cite{NiWu}, we studied a new type of genus bounds. Roughly speaking, given a torsion class $a\in H_1(Y;\mathbb Z)$, where $Y$ is a $3$--manifold, one can consider the minimal rational genus of all knots representing $a$. This defines a function $\Theta$ on the torsion subgroup of $H_1(Y)$, which was introduced by Turaev \cite{TuFunc} as an analogue of Thurston norm \cite{Th}. When the homology class $a$ has order $2$, this $\Theta$ is essentially the minimal genus of embedded nonorientable surfaces in a given $\mathbb Z_2$--homology class. To state the theorem, let us recall that the correction term of a rational homology sphere $Y$ with a Spin$^c$ structure $\mathfrak s\in \spin(Y)$ is a rational number $d(Y,\mathfrak s)$. There is an affine action of $H^2(Y;\mathbb Z)$ on $\spin(Y)$ which is denoted by addition.

The following theorem is essentially \cite[Theorem~3.3]{GreeneNi}, which is an easy corollary of the main theorem in \cite{NiWu}.

\begin{thm}\label{thm:NOgenus}
Let $Y$ be a rational homology $3$--sphere, $$\beta\co H_2(Y;\mathbb Z_2)\to H_1(Y;\mathbb Z)$$ be the Bockstein homomorphism, and $\mathrm{PD}\co H_i(Y)\to H^{3-i}(Y)$ be the Poincar\'e duality map. Suppose that a nonzero class $A\in H_2(Y;\mathbb Z_2)$ is represented by a closed connected nonorientable surface of genus $h$, then
$$h\ge2\max_{\mathfrak s\in\spin(Y)}\big\{d(Y,\mathfrak s+\mathrm{PD}\circ\beta(A))-d(Y,\mathfrak s)\big\}.$$
\end{thm}

A stronger version of the above theorem is obtained by Levine--Ruberman--Strle \cite{LevRubStr}, who proved a nonorientable genus bound in dimensional four.

Bredon and Wood \cite{BW} initiated the study of minimal genus nonorientable surfaces representing a given $\mathbb Z_2$--homology class. They have completely determined the minimal genus for lens spaces $L(2k,q)$. (Lens spaces with odd order do not contain embedded nonorientable surfaces, since such surfaces always represent nontrivial $\mathbb Z_2$--homology classes.) This problem is closely related to one-sided Heegaard splittings introduced by Rubinstein \cite{RubinOS}, since the Heegaard surfaces in this case are nonorientable. As a consequence, Theorem~\ref{thm:NOgenus} gives a lower bound to the one-sided Heegaard genus of a rational homology sphere in terms of Heegaard Floer correction terms, while such kind of bounds are not known for the usual two-sided Heegaard genus.

Nonorientable genus is also closely related to the $\mathbb Z_2$--Thurston norm, which is discussed in Section~\ref{sect:RatGenus}.

\begin{cor}\label{cor:Z2Norm}
If an irreducible rational homology $3$--sphere $Y$ satisfies the condition
\begin{equation}\label{eq:SurjZ2}
\scriptstyle H_1(S;\mathbb Z_2)\to H_1(Y;\mathbb Z_2)  \text{ \small is surjective for any orientable incompressible surface }S\subset Y,
\end{equation}
then the $\mathbb Z_2$--Thurston norm of any $A\in H_2(Y;\mathbb Z_2)$ is bounded below by $$-2+2\max_{\mathfrak s\in\spin(Y)}\big\{d(Y,\mathfrak s+\mathrm{PD}\circ\beta(A))-d(Y,\mathfrak s)\big\}.$$
\end{cor}

Given a three-manifold  $Y$, the {\it complexity} $C(Y)$ of $Y$ is the minimal number of tetrahedra one needs to triangulate $Y$. This invariant is notoriously hard to compute. Jaco--Rubinstein--Tillmann \cite{JRT} found a lower bound to $C(Y)$ for some $3$--manifolds in terms of the $\mathbb Z_2$--Thurston norm. Using Corollary~\ref{cor:Z2Norm} and its variants, we can compute the $\mathbb Z_2$--Thurston norm for some $3$--manifolds, thus give lower bounds to $C(Y)$. Upper bounds to $C(Y)$ can be obtained by constructing layered-triangulations \cite{JR}. We carry out this computation/construction explicitly for two classes of manifolds.

\begin{prop}\label{prop:CBraid2n}
Let $L$  be the closure of the braid $$\sigma=\sigma_1\sigma_2^{-2a_1}\sigma_1\sigma_2^{-2a_2}\cdots
\sigma_1\sigma_2^{-2a_{2n-1}}\sigma_1\sigma_2^{-2a_{2n}},$$ where $a_i,n>0$, $\Sigma(L)$ be the double branched cover of $S^3$ branched over $L$. Then the complexity $C(\Sigma(L))$ of $\Sigma(L)$ is in the range
$$2n-4+2\sum_{i=1}^{2n}a_i\le C(\Sigma(L))\le 4n+2\sum_{i=1}^{2n}a_i.$$
\end{prop}

\begin{prop}\label{prop:CBraid3}
Let $L$  be the closure of the braid $\sigma_1\sigma_2^{-2a-1}\sigma_1\sigma_2^{-2b-1}\sigma_1\sigma_2^{-2c-1}$, where $a,b,c$ are nonnegative integers.  Then the complexity $C(\Sigma(L))$ of $\Sigma(L)$ is in the range
$$2(a+b+c)+2\le C(\Sigma(L))\le 2(a+b+c)+9.$$
\end{prop}

This paper is organized as follows:
In Section~\ref{sect:RatGenus}, we review our earlier work on rational genus bounds. Our Theorem~\ref{thm:NOgenus} is an immediate corollary of this work. We then show that this bound gives lower bounds to one-sided Heegaard genus and $\mathbb Z_2$--Thurston norm.
In Section~\ref{sect:DoubCov}, we review Ozsv\'ath and Szab\'o's algorithm of computing the correction terms of the double branched cover of $S^3$ branched over alternating links. In Section~\ref{sect:Comp}, we carry out the computation for the double branched cover of $S^3$ branched over some alternating closed $3$--braids, and determine their $\mathbb Z_2$--Thurston norm. In Section~\ref{sect:G1OB}, we construct layered triangulations for manifolds admitting a genus one open book decomposition with connected binding, and give upper and lower bounds for the complexity of the manifolds we consider in Section~\ref{sect:Comp}.

\vspace{5pt}\noindent{\bf Acknowledgements.}\quad We wish to thank Ian Agol, Danny Ruberman and Hyam Rubinstein for conversations which motivated this work. The first author was
partially supported by NSF grant
numbers DMS-1103976, DMS-1252992, and an Alfred P. Sloan Research Fellowship. The second author was partially supported by grants from the Research Grants Council of the Hong Kong Special Administrative Region, China (Project No. CUHK 2191056).


\section{Rational genus and non-orientable genus bounds}\label{sect:RatGenus}

\subsection{Rational genus bounds in Heegaard Floer homology}

Heegaard Floer homology, introduced by
Ozsv\'ath and Szab\'o \cite{OSzAnn1}, is an invariant  for closed oriented Spin$^c$ $3$--manifolds $(Y,\mathfrak s)$,
  taking the form of a collection of related homology groups as  $\widehat{HF}(Y,\mathfrak s)$, $HF^{\pm}(Y,\mathfrak s)$, and $HF^\infty(Y,\mathfrak s)$.
There is a $U$--action on Heegaard Floer homology groups.
When $\mathfrak s$ is torsion, there is an absolute Maslov $\mathbb Q$--grading on the Heegaard Floer homology groups. The $U$--action decreases the grading by $2$.

For a rational homology $3$--sphere $Y$ with a Spin$^c$ structure $\mathfrak s$, $HF^+(Y,\mathfrak s)$ can be decomposed as the direct sum of two groups: the first group is the image of $HF^\infty(Y,\mathfrak s)\cong\mathbb Z[U,U^{-1}]$ in $HF^+(Y,\mathfrak s)$,
which is isomorphic to $\mathcal T^+=\mathbb Z[U,U^{-1}]/UZ[U]$, and its minimal absolute  $\mathbb{Q}$--grading is an invariant of $(Y,\mathfrak s)$, denoted by $d(Y,\mathfrak s)$, the {\it correction term} \cite{OSzAbGr}; the second group is the quotient modulo the above image and is denoted by $HF_{\mathrm{red}}(Y,\mathfrak s)$.  Altogether, we have $$HF^+(Y,\mathfrak s)=\mathcal{T}^+\oplus HF_{\mathrm{red}}(Y,\mathfrak s).$$

Suppose that $Y$ is a closed oriented $3$--manifold, there is a kind of ``norm'' function one can define on the torsion subgroup of $H_1(Y;\mathbb Z)$. To define it, let us first recall  the rational genus of a rationally null-homologous knot $K\subset Y$ defined by Calegari and Gordon \cite{CG}.

Suppose that $K$ is a rationally null-homologous oriented knot in $Y$, and $\nu(K)$ is a tubular neighborhood of $K$. A properly embedded oriented connected surface $F\subset Y\backslash\overset{\circ}{\nu}(K)$ is called a {\it rational Seifert surface} for $K$, if $\partial F$ consists of coherently oriented parallel curves on $\partial\nu(K)$, and the orientation of $\partial F$ is coherent with the orientation of $K$. The {\it rational genus} of $K$ is defined to be
$$g_r(K)=\min_{F}\frac{\max\{0,-\chi(F)\}}{2|[\mu]\cdot[\partial F]|},$$
where $F$ runs over all the rational Seifert surfaces for $K$, and $\mu\subset\partial\nu(K)$ is the meridian of $K$.

The rational genus is a natural generalization of the genus of null-homologous knots. Moreover, given a torsion class in $H_1(Y)$, one can consider the minimal rational genus for all knots in this torsion class. More precisely,
given $a\in\mathrm{Tors} H_1(Y)$, let
$$\Theta(a)=\min_{K\subset Y,\:[K]=a}2g_r(K).$$
This $\Theta$ was introduced by Turaev \cite{TuFunc} in a slightly different form. Turaev regarded $\Theta$ as an analogue of Thurston norm \cite{Th}, in the sense that it measures the minimal normalized Euler characteristic of a ``folded surface'' representing a given class in $H_2(Y;\mathbb Q/\mathbb Z)$.

The main result in \cite{NiWu} gives a lower bound to $\Theta$ via Heegaard Floer correction terms.

\begin{thm}[Ni--Wu]\label{thm:GenusBound}
Suppose that $Y$ is a rational homology $3$--sphere, $K\subset Y$ is a knot, $F$ is a rational Seifert surface for $K$. Then
\begin{equation}\label{eq:GenusBound}
1+\frac{-\chi(F)}{|[\partial F]\cdot [\mu]|}\ge \max_{\mathfrak s\in\spin(Y)}\big\{d(Y,\mathfrak s+\mathrm{PD}[K])-d(Y,\mathfrak s)\big\}.
\end{equation}
The right hand side of (\ref{eq:GenusBound}) only depends on the manifold $Y$ and the homology class of $K$, so it gives a lower bound of $1+\Theta(a)$ for the homology class $a=[K]$.
\end{thm}

Theorem~\ref{thm:NOgenus} is an immediate corollary of Theorem~\ref{thm:GenusBound}, as explained in \cite[Theorem~3.3]{GreeneNi}.


\subsection{Nonorientable genus and one-sided Heegaard splittings}

Theorem~\ref{thm:NOgenus} can be used to study one-sided Heegaard splitting, introduced in \cite{RubinOS}.

\begin{defn}
Let $Y$ be a closed orientable $3$--manifold. A pair
$(Y,\Pi)$ is called a {\it one-sided Heegaard splitting} if $\Pi\subset Y$ is a closed nonorientable
surface, such that its complement is an open handlebody. Moreover, $[\Pi]\in H_2(Y;\mathbb Z_2)$ is called the class {\it associated with} $(Y,\Pi)$.
\end{defn}

Rubinstein \cite{RubinOS} proved that for any nonzero element $A\in H_2(Y;\mathbb Z_2)$, there exists a one-sided Heegaard splitting $(Y,\Pi)$ such that $[\Pi]=A$.
He also studied when there exists an incompressible one-sided Heegaard splitting.

\begin{thm}[Rubinstein]\label{thm:IncompOS}
If $Y$ is irreducible and $b_1(Y)=0$, then the condition (\ref{eq:SurjZ2}) is equivalent to the condition that the complement of any nonorientable incompressible surface is an open handlebody.
Under this condition, there is an incompressible one-sided splitting associated with any nonzero class.
\end{thm}

When $Y$ is the lens space $L(2k,q)$, Rubinstein \cite{RubinOS} proved that all incompressible surfaces in $Y$ are isotopic, and the genus $N(2k,q)$ can be computed by \cite{BW}. Recently, Johnson \cite{Johnson} proved that any nonorientable surface in
$L(2k,q)$ with a given genus is unique up to isotopy, thus $L(2k,q)$ has a unique genus $g$ one-sided Heegaard splitting for any given genus $g\ge N(2k,q)$.

One can define the {\it one-sided Heegaard genus} of $(Y,A)$ to be the minimal genus of $\Pi$ such that $(Y,\Pi)$ is a one-sided Heegaard splitting and $[\Pi]=A$.
Theorem~\ref{thm:NOgenus} clearly gives a lower bound to the one-sided Heegaard genus.


\subsection{A lower bound to the $\mathbb Z_2$--Thurston norm}

Theorem~\ref{thm:NOgenus} may also be used to give a lower bound to the $\mathbb Z_2$--Thurston norm introduced in \cite{JRT}.

\begin{defn}
Given a closed, not necessarily orientable, surface $S$, let $$\chi_-(S)=\sum_{S_i\subset S}\max\{-\chi(S_i),0\},$$
where the sum is taken over all the components of $S$. For any $A\in H_2(Y;\mathbb Z_2)$, define the {\it $\mathbb Z_2$--Thurston norm} of $A$ to be
$$||A||:=\min\{\chi_-(S)|[S]=A\}.$$
A surface $S$ is {\it $\mathbb Z_2$--taut}, if no component of $S$ is a sphere or $\mathbb RP^2$ and $||[S]||=-\chi(S)$.
\end{defn}

Clearly, if $S$ is $\mathbb Z_2$--taut, then each component of $S$ is non-separating and incompressible.

\begin{proof}[Proof of Corollary~\ref{cor:Z2Norm}]
By Theorem~\ref{thm:IncompOS}, any nonorientable incompressible surface must be connected. Our conclusion then follows from Theorem~\ref{thm:NOgenus}.
\end{proof}

\begin{rem}
The condition (\ref{eq:SurjZ2}) in Corollary~\ref{cor:Z2Norm} is used to ensure that any $\mathbb Z_2$--taut surface is connected. This condition may be removed if we can prove an analogue of Theorem~\ref{thm:GenusBound} for links.
\end{rem}

When the rank of $H_2(Y;\Z_2)$ is small, it is easy to remove the condition (\ref{eq:SurjZ2}). For example, if $H_2(Y;\Z_2)\cong \Z_2$, any $\mathbb Z_2$--taut surface is necessarily connected. If $H_2(Y;\Z_2)\cong \Z_2^2$, we may use the following lemma.

\begin{lem}\label{lem:TautConn}
Suppose that $H_2(Y;\Z_2)\cong\Z_2^2$, and the three nonzero elements in $H_2(Y;\Z_2)$ are $\alpha_1,\alpha_2,\alpha_3$. Let $h_i$ be a lower bound to the genus of the closed {\bf connected} nonorientable surfaces representing $\alpha_i$, $i=1,2,3$. If $h_i$'s satisfy the inequalities
$$h_i+h_{i+1}\ge h_{i+2}+2,\quad i=1,2,3,$$
then $h_i-2\le||\alpha_i||$.
\end{lem}
\begin{proof}
Let $S$ be a $\mathbb Z_2$--taut surface representing $\alpha_1$. If $S$ is connected, our conclusion obviously holds. If $S$ is disconnected, we may assume that no component of $S$ is null-homologous in $H_2(Y;\Z_2)$, and any two components of $S$ are not homologous in $H_2(Y;\Z_2)$. If follows that $S$ has exactly two components, and they represent $\alpha_2$ and $\alpha_3$, respectively. So $\chi_-(S)\ge (h_2-2)+(h_3-2)\ge h_1-2$. Hence $||\alpha_1||\ge h_1-2$. Similarly, $h_i-2\le ||\alpha_i||$ for $i=2,3$.
\end{proof}

Jaco, Rubinstein and Tillmann \cite{JRT} showed that the $\mathbb Z_2$--Thurston norm gives a lower bound to the complexity of some manifolds.

\begin{thm}[Jaco--Rubinstein--Tillmann]\label{thm:JRT}
Let $Y$ be a closed, orientable, irreducible,
atoroidal, connected $3$--manifold with triangulation $\mathscr T$. Let $H\subset H_2(Y;\mathbb Z_2)$ be a rank $2$ subgroup, then
$$|\mathscr T|\ge 2+\sum_{A\in H}||A||.$$
\end{thm}

\begin{cor}
Suppose that $Y$ is irreducible and atoroidal, $b_1(Y)=0$, and $Y$ satisfies (\ref{eq:SurjZ2}).  Let $H\subset H_2(Y;\mathbb Z_2)$ be a rank $2$ subgroup, then the complexity of $Y$ is bounded below by
$$-4+2\sum_{A\in H}\max_{\mathfrak s\in\spin(Y)}\big\{d(Y,\mathfrak s+\mathrm{PD}\circ\beta(A))-d(Y,\mathfrak s)\big\}.$$
\end{cor}
\begin{proof}
This follows from Corollary~\ref{cor:Z2Norm} and Theorem~\ref{thm:JRT}.
\end{proof}


\section{The double branched cover of alternating links}\label{sect:DoubCov}

For any link $L\subset S^3$, let $\Sigma(L)$ be the double branched cover of $S^3$ branched along $L$. In this section, we will review
an algorithm of Ozsv\'ath and Szab\'o \cite{OSzAbGr}, which computes the correction terms of $\Sigma(L)$ when $L$ is a non-split alternating link.

Let $L$ be a link with a connected diagram.
The diagram separates the plane into several regions. We can color the regions by black and white, such that any two adjacent regions have different colors.
When the diagram is alternating, we can choose the coloring convention as in Figure~\ref{fig:BW}.

\begin{figure}[ht]
\begin{picture}(340,80)
\put(140,0){\includegraphics[scale=0.8]{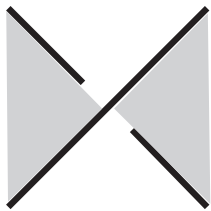}}

\end{picture}
\caption{\label{fig:BW}The coloring convention}
\end{figure}

We now define the black graph $\mathcal B(L)$ of the diagram as follows: the vertices of $\mathcal B(L)$ are in one-to-one correspondence with the black regions,
and the edges of $\mathcal B(L)$ are in one-to-one correspondence with the crossings. For each crossing between two regions $R_1,R_2$, there is an edge connecting the corresponding vertices $v_1,v_2$.

Now we choose a maximal subtree $T$ of $\mathcal B(L)$, and let $\{e_i\}_{i=1}^b$ be the edges in $Z_T=\mathcal B(L)-T$. We orient these $e_i$ in any way we like.
Since $T$ is a maximal subtree, for any $e_i\in Z_T$ there is a unique circuit $C_i$ in $T\cup \{e_i\}$, with the orientation coherent with $e_i$.

Let $V=\mathbb Z^b$ be generated by $e_1,e_2,\dots,e_b$.
We define a quadratic form $Q\co V\otimes V\to \mathbb Z$ as follows.
For any $e_i,e_j$,
$$Q(e_i\otimes e_j)=\pm E(C_i\cap C_j).$$
Here $E(C_i\cap C_j)$ is the number of edges in $C_i\cap C_j$, the sign is positive if the orientations of $C_i$ and $C_j$ are different on $C_i\cap C_j$, and negative if the orientations of $C_i$ and $C_j$ are the same on $C_i\cap C_j$.
In particular, $Q(e_i\otimes e_i)=-E(C_i)$.

Let $V^*=\mathrm{Hom}(V,\Z)$ be the dual group of $V$.
For any $\alpha\in V^*$, define
\begin{equation}\label{eq:DualNorm}
|\alpha|^2=\max_{v\in V\otimes\mathbb R-\{0\}}\frac{(\alpha(v))^2}{Q(v,v)}.
\end{equation}
We define a homomorphism $q\co V\to V^*$ by letting
$$(q(a))(v)=Q(a,v).$$

\begin{notn}
Under the basis $\{e_1,e_2,\dots,e_b\}$, elements in $V$ can be represented by column vectors. Let $\{\alpha_1,\alpha_2,\dots,\alpha_b\}$ be the dual basis of $V^*$, namely, $\alpha_i(e_j)=\delta_{ij}$, then elements in $V^*$ will be represented by row vectors.
\end{notn}

\begin{lem}
For $\alpha=\sum a_i\alpha_i$, let $\tensor a=(a_1,a_2,\dots,a_b)$. Suppose that $Q$ is represented by the symmetric matrix $\tensor Q$, then
\begin{equation}\label{eq:||^2}
|\alpha|^2=\tensor a\tensor Q^{-1}\tensor a^T.
\end{equation}
\end{lem}
\begin{proof}
For any $X_0\in V\otimes \mathbb R$, the tangent hyperplane of the ellipsoid
$$X^T\tensor Q X=X_0^T\tensor QX_0$$
at $X_0$ is
$$X_0^T\tensor Q X=X_0^T\tensor QX_0.$$
On the other hand,
the hyperplane defined by $\alpha$ is
$$\tensor aX=\tensor aX_0.$$
These two hyperplanes coincide when
$$\tensor a=X_0^T\tensor Q.$$
So
\begin{eqnarray*}
|\alpha|^2&=&\left.\frac{|\alpha(X_0)|^2}{Q(X_0,X_0)}\right|_{X_0=(\tensor a\tensor Q^{-1})^T}\\
&=&\frac{|\tensor a(\tensor a\tensor Q^{-1})^T|^2}{(\tensor a\tensor Q^{-1})\tensor Q(\tensor a\tensor Q^{-1})^T}\\
&=&\tensor a\tensor Q^{-1}\tensor a^T.
\end{eqnarray*}
\end{proof}

An element $\kappa\in V^*$ is {\it characteristic} if
$$\kappa(v)\equiv Q(v,v)\pmod2,\qquad\text{for any $v\in V$}.$$
Let $\mathcal C\subset V^*$ be the set of all characteristic elements.
Two characteristic elements $\kappa_1,\kappa_2$ are {\it equivalent} if
$$\kappa_1-\kappa_2=2q(v),\qquad\text{for some $v\in V$}.$$
Let $C_1(V)=\mathcal C/(2q(V))$ be the set of all equivalent classes of characteristic elements,
which is an affine space over the group $H=V^*/(q(V))$.
More precisely, given $c\in C_1(V)$ and $[\alpha]\in H$, $c+2[\alpha]$ is a well-defined element in $C_1(V)$, and $c+2[\alpha]=c$ if and only if $[\alpha]=0$.

For any $c\in C_1(V)$, define
\begin{equation}\label{eq:Corr}
d(V,c)=\max_{\kappa\in c}\frac{|\kappa|^2+b}4.
\end{equation}

\begin{thm}[Ozsv\'ath--Szab\'o]
Suppose that $Y=\Sigma(L)$, where $L$ is an alternating link,
let $Q\co V\otimes V\to\mathbb Z$ be a quadratic form constructed as before,
then $C_1(V)=\spin(Y)$ and $d(V,c)$ is the corresponding correction term of $Y$.
\end{thm}


\section{Computation}\label{sect:Comp}

In this section, we use Theorem~\ref{thm:NOgenus} to determine the $\Z_2$-Thurston norm of the double branched cover of some alternating closed 3-braids.

Let $\sigma_1$, $\sigma_2$ be the two standard generators of $B_3$.  We summarize our results below.

\begin{prop}\label{prop:Braid2n}
Let $L$  be the closure of the braid $$\sigma=\sigma_1\sigma_2^{-2a_1}\sigma_1\sigma_2^{-2a_2}\cdots
\sigma_1\sigma_2^{-2a_{2n-1}}\sigma_1\sigma_2^{-2a_{2n}},$$ where $a_i,n>0$, then the $\Z_2$--Thurston norms of the three nonzero
homology classes in $H_2(\Sigma(L);\Z_2)$ are
$$\sum_{i\text{ odd}}a_{i}+n-2,\quad \sum_{i\text{ even}}a_{i}+n-2,\quad \sum_{i=1}^{2n}a_{i}-2.$$
\end{prop}

\begin{prop}\label{prop:Braid3}
Let $L$  be the closure of the braid $\sigma_1\sigma_2^{-2a-1}\sigma_1\sigma_2^{-2b-1}\sigma_1\sigma_2^{-2c-1}$, where $a,b,c$ are nonnegative integers.  Then the $\Z_2$-Thurston norms of the three nonzero homology classes in $H_2(\Sigma(L);\Z_2)$ are $a+b$, $b+c$ and $c+a$.
\end{prop}

We point out that each of the family considered above consists of pure braids only.  Hence, $H_2(\Sigma(L);\Z_2)\cong \Z_2^2$ and there are exactly three nonzero $\Z_2$-homology classes.

\subsection{Proof of Proposition~\ref{prop:Braid2n}}

\noindent{\bf Case 1.} $n=1$, namely, $\sigma=\sigma_1\sigma_2^{-2a}\sigma_1\sigma_2^{-2b}$, $a,b>0$.

\begin{figure}[ht]
\centering
\includegraphics[scale=0.3]{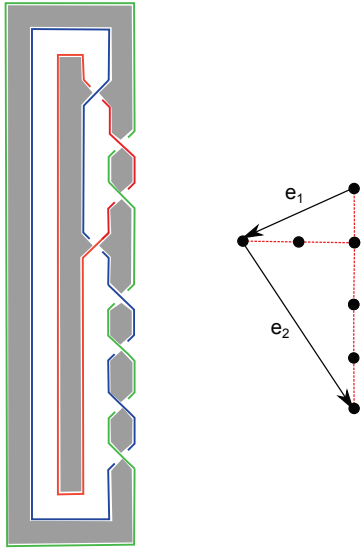}
\caption{The closure of the pure braid $\sigma=\sigma_1\sigma_2^{-2a}\sigma_1\sigma_2^{-2b}$ ($a=1, b=2$ here) and its black graph on the right.  The dashed red edges indicate a maximal subtree $T$, and the remaining two edges are oriented.}
\label{family1}
\end{figure}

Using the notation from the previous section, we construct the black graph $\mathcal{B}(L)$ in Figure~\ref{family1}.  Choose a maximal subtree $T$, then $Z_T$ has two edges $e_1, e_2$.  So $V=\langle e_1,e_2 \rangle $.  The quadratic form is represented by the symmetric matrix $$\tensor Q=
                     \begin{pmatrix}
                       -2a-2 & 2 \\
                       2 & -2b-2 \\
                     \end{pmatrix}
                     $$
The set of characteristic elements is
$$\mathcal{C}=\{a_1\alpha_1+a_2\alpha_2 \,| \, a_1\equiv a_2 \equiv 0 \pmod2 \}.$$

Let $$\kappa_0=(0,0),\;\; \;\kappa_1=q(e_1)=(-2a-2,2),$$  $$\kappa_2=q(e_2)=(2,-2b-2), \;\;\;\kappa_3=q(e_1+e_2)=(-2a,-2b)$$ be characteristic elements in the affine space $\mathcal{C}_1(V)$.    Straightforward computation yields
$$|\kappa_0|^2=\kappa_0\tensor Q^{-1}\kappa_0^T=0, $$
$$|\kappa_1|^2=\kappa_1\tensor Q^{-1}\kappa_1^T=-2(1+a),  $$
$$|\kappa_2|^2=\kappa_2\tensor Q^{-1}\kappa_2^T=-2(1+b),  $$
$$|\kappa_3|^2=\kappa_3\tensor Q^{-1}\kappa_3^T=-2(a+b).  $$

\begin{lem}
Each of the characteristic elements $\kappa_i$ maximizes the above quadratic form within its equivalent class of characteristic elements, i.e.,
$$|\kappa_i+2q(v)|^2\leq |\kappa_i|^2$$ for any $v\in V$, $i=0,1,2,3$.
\end{lem}

\begin{proof}
The characteristic elements equivalent to $\kappa_i$ have the form
$$\kappa=q(me_1+ne_2), \;\; m,n\in \Z$$
where
\begin{itemize}
\item $m,n$ even, when $i=0$
\item $m$ odd $n$ even, when $i=1$
\item $m$ even $n$ odd, when $i=2$
\item $m, n$ odd, when $i=3$
\end{itemize}
and
\begin{eqnarray*}
|\kappa|^2 &=&\kappa\tensor Q^{-1}\kappa^T \\
&=&(m,n)\tensor Q\begin{pmatrix}m\\n\end{pmatrix}\\
&=&-(2a+2)m^2-(2b+2)n^2+4mn\\
&=&-2((m-n)^2+am^2+bn^2).
\end{eqnarray*}
Since $a,b>0$, it is clear that each $|\kappa_i|^2$ is a maximizer for the given parity on $m$ and $n$.
\end{proof}

Hence, we can compute the correction terms via the formula (\ref{eq:Corr}) and apply Theorem~\ref{thm:NOgenus} and Lemma~\ref{lem:TautConn} to get a lower bound to the $\Z_2$-Thurston norm of $\Sigma(L)$.  Note that $\kappa_1-\kappa_0,\kappa_2-\kappa_0,\kappa_3-\kappa_0$ represent the three different homology classes of order 2 in $H^2(Y)$. For any $\alpha\in H_2(Y;\Z_2)$, let $h(\alpha)$ be the minimal genus of closed connected non-orientable surfaces representing $\alpha$. We obtain:
$$h(\beta^{-1}\mathrm{PD}^{-1}(\kappa_1-\kappa_0))\geq 2(d(Y,[\kappa_0])-d(Y,[\kappa_1]))=a+1.$$
$$h(\beta^{-1}\mathrm{PD}^{-1}(\kappa_2-\kappa_0))\geq 2(d(Y,[\kappa_0])-d(Y,[\kappa_2]))=b+1.$$
$$h(\beta^{-1}\mathrm{PD}^{-1}(\kappa_3-\kappa_0))\geq 2(d(Y,[\kappa_0])-d(Y,[\kappa_3]))=a+b.$$
Thus by Lemma~\ref{lem:TautConn}, the corresponding $\Z_2$-Thurston norms are bounded below by $$a-1,b-1,a+b-2.$$

Finally, we show that the lower bounds here are sharp.  Indeed, the lift of the disk bounded by each component of the link in $\Sigma(L)$ has Euler characteristic exactly $1-a$, $1-b$ and $2-a-b$, respectively.  (For example, the disk bounded by the green curve in Figure~\ref{family1} intersects with the red and the blue curves at $a+b$ points in total.  By the Riemann--Hurwitz formula, the Euler characteristic of its double branched cover is $2\chi(D^2)-(a+b)=2-a-b$.)

\

\noindent{\bf Case 2.}
 $n>1$.

\begin{figure}[ht]
\centering
\scalebox{0.35}{\includegraphics*[0,200][590,680]{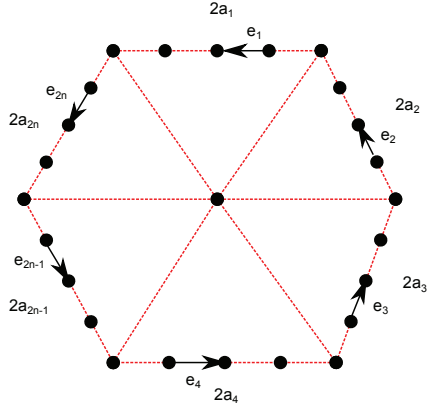}}
\caption{The black graph of the closure of $\sigma$ in this case. The graph is supported in a $2n$--gon $P$. The $i$th side of $P$ consists of $2a_i$ edges, and there is an edge connecting the center of $P$ with each corner vertex of $P$. The maximal tree $T$ is obtained by removing an edge from each side of $P$.}
\label{fig:family2}
\end{figure}

This case is very similar to the previous case.  The quadratic form $Q$ is represented by the symmetric matrix
$$\tensor Q=
    \begin{pmatrix}
      -2a_1-2 & 1 & 0 &\cdots &0 & 1 \\
      1 & -2a_2-2 & 1 &\cdots &0 & 0 \\
      0 & 1 & -2a_3-2 &\cdots &0 &0 \\
     \vdots &\vdots &\vdots &\ddots &\vdots&\vdots\\
     0 &0 &0 &\cdots &-2a_{2n-1}-2 &1\\
      1 & 0 &0 &\cdots & 1 & -2a_{2n}-2 \\
    \end{pmatrix}.
$$
Given $X=\sum_{i=1}x_ie_i$, we have
$$Q(X,X)=X^T\tensor QX=-\sum_{i=1}^{2n}\big(2a_ix_i^2+(x_i-x_{i+1})^2\big),$$
where the subscripts are understood modulo $2n$.

The set of characteristic elements is
$$\mathcal{C}=\Big\{\sum_{i=1}^{2n}a_i\alpha_i\,\Big| \, a_i\equiv 0 \pmod 2 \Big\},$$ and $q(e_i)=(\tensor Qe_i)^T$ is the $i$th row of $\tensor Q$.

Let $$\kappa_0=\tensor 0\in V^*,$$
  $$\kappa_1=q(\sum_{i\text{ odd}}e_i)=(-2a_1-2,2,-2a_3-2,2,\cdots,-2a_{2n-1}-2,2),$$
$$\kappa_2=q(\sum_{i\text{ even}}e_i)=(2,-2a_2-2,2,-2a_4-2,\cdots,2,-2a_{2n}-2),$$
  $$\kappa_3=\kappa_1+\kappa_2=(-2a_1,-2a_2,-2a_3,-2a_4,\cdots,-2a_{2n-1},-2a_{2n})$$ be characteristic elements in the affine space $\mathcal{C}_1(V)$.    Straightforward computation yields
$$|\kappa_0|^2=\kappa_0\tensor Q^{-1}\kappa_0^T=0, $$
$$|\kappa_1|^2=\kappa_1\tensor Q^{-1}\kappa_1^T=\kappa_1\sum_{i\text{ odd}}e_i=-2(n+\sum_{i\text{ odd}}a_i),  $$
$$|\kappa_2|^2=\kappa_2\tensor Q^{-1}\kappa_2^T=\kappa_2\sum_{i\text{ even}}e_i=-2(n+\sum_{i\text{ even}}a_i),  $$
$$|\kappa_3|^2=\kappa_3\tensor Q^{-1}\kappa_3^T=\kappa_3(\sum_{i=1}^{2n}e_i)=-2\sum_{i=1}^{2n}a_i.  $$

\begin{lem}
Each of the characteristic element $\kappa_j$ maximizes the above quadratic form within its equivalent class of characteristic elements, i.e.,
$$|\kappa_j+2q(v)|^2\leq |\kappa_j|^2$$ for any $v\in V$, $j=0,1,2,3$.
\end{lem}

\begin{proof}
The characteristic elements equivalent to $k_j$ have the form
$$\kappa=q(\sum_{i=1}^{2n}x_ie_i), \;\; x_i\in \Z$$
where
\begin{itemize}
\item $x_i$ even for any $i$, when $j=0$
\item $x_i\equiv i \pmod2$, when $j=1$
\item $x_i\equiv i+1 \pmod2$, when $j=2$
\item $x_i$ odd for any $i$, when $j=3$
\end{itemize}
and
\begin{eqnarray*}
|\kappa|^2&=&(\tensor QX)^T\tensor Q^{-1}(\tensor QX)\\
&=&X^T\tensor QX\\
&=&-\sum_{i=1}^{2n}(2a_ix_i^2+(x_i-x_{i+1})^2).
\end{eqnarray*}
Since $a_i>0$, it is clear that each $|\kappa_i|^2$ is a maximizer for the given parity on $x_i$'s.
\end{proof}

Similar to the previous case, $\kappa_1-\kappa_0,\kappa_2-\kappa_0,\kappa_3-\kappa_0$ represent the three different homology classes of order 2 in $H^2(Y)$, and
$$h(\beta^{-1}\mathrm{PD}^{-1}(\kappa_1-\kappa_0))\geq 2(d(Y,[\kappa_0])-d(Y,[\kappa_1]))=\sum_{i\text{ odd}}a_{i}+n.$$
$$h(\beta^{-1}\mathrm{PD}^{-1}(\kappa_2-\kappa_0))\geq 2(d(Y,[\kappa_0])-d(Y,[\kappa_2]))=\sum_{i\text{ even}}a_{i}+n.$$
$$h(\beta^{-1}\mathrm{PD}^{-1}(\kappa_3-\kappa_0))\geq 2(d(Y,[\kappa_0])-d(Y,[\kappa_1]))=\sum_{i=1}^{2n}a_{i}.$$
By Lemma~\ref{lem:TautConn}, we get lower bounds to the $\Z_2$-Thurston norm.
The lower bounds here are also sharp - the lift of the disks bounded by each component of the link in $\Sigma(L)$ has negative Euler characteristic $$\sum_{i\text{ odd}}a_{i}+n-2,\quad \sum_{i\text{ even}}a_{i}+n-2,\quad \sum_{i=1}^{2n}a_{i}-2,$$ respectively.


\subsection{Proof of Proposition~\ref{prop:Braid3}}

The quadratic form $Q$ is represented by the symmetric matrix
$$\tensor Q=
      \begin{pmatrix}
        -2a-3 & 1 & 1 \\
        1 & -2b-3 & 1 \\
        1 & 1 & -2c-3 \\
      \end{pmatrix}
$$

The set of characteristic elements
$$\mathcal{C}=\{a_1\alpha_1+a_2\alpha_2+a_3\alpha_3\,| \, a_1\equiv a_2 \equiv a_3 \equiv 1 \pmod 2 \},$$ and $$ q(e_1)=(-2a-3,1,1), \; q(e_2)=(1,-2b-3,1), \; q(e_3)=(1,1, -2c-3).         $$

Let $$\kappa_0=(1,-1,1), \; \kappa_1=\kappa_0+q(e_1-e_2)=(-2a-3,2b+3,1)$$
be characteristic elements in the affine space $\mathcal{C}_1(V)$.  We claim:

\begin{lem}
Each of the characteristic elements $\kappa_i$ maximizes the quadratic form within its equivalent class of characteristic elements, i.e.,
$$|\kappa_i+2q(v)|^2\leq |\kappa_i|^2$$ for any $v\in V$, $i=0,1$.
\end{lem}

\begin{proof}
The characteristic elements equivalent to $\kappa_i$ have the form
$$\kappa=\kappa_0+q(x_1e_1+x_2e_2+x_3e_3), \;\; x_i\in \Z$$
where
\begin{itemize}
\item $x_1,x_2,x_3$ even, when $i=0$
\item $x_1, x_2$ odd, $x_3$ even, when $i=1$
\end{itemize}
and
\begin{eqnarray*}
|\kappa|^2 &=&|\kappa_0|^2-(2a+3)x_1^2-(2b+3)x_2^2-(2c+3)x_3^2\\
&&\, +2x_1x_2+2x_1x_3+2x_2x_3+2x_1-2x_2+2x_3\\
&=&|\kappa_0|^2-\Big(2ax_1^2+2bx_2^2+2cx_3^2+\sum_i(x_i-x_{i+1})^2\\
&&\qquad\qquad +(x_1-1)^2+(x_2+1)^2+(x_3-1)^2-3\Big).
\end{eqnarray*}
Since $a,b,c \geq 0$, it is clear that each $|\kappa_j|^2$ is a maximizer for the given parity on $x_i$'s.
\end{proof}

Note that $\kappa_1-\kappa_0$ represents one of the homology classes of order 2 in $H^2(Y)$.  Hence,
$$h(\beta^{-1}\mathrm{PD}^{-1}(\kappa_1-\kappa_0))\geq 2(d(Y,[\kappa_0])-d(Y,[\kappa_1]))=a+b+2.$$
Similar arguments apply to the other two homology classes of order 2, we get lower bounds to $h$ given by $b+c+2$ and $c+a+2$.
By Lemma~\ref{lem:TautConn}, we get lower bounds to the $\Z_2$-Thurston norms given by
$$a+b,\quad b+c,\quad c+a.$$
Meanwhile, the lower bounds are sharp - the lift of one of the disks bounded by a component of the link in $\Sigma(L)$ has the desired Thurston norm.


\section{Genus-one open books and layered-triangulations}\label{sect:G1OB}

In this section, we will construct layered-triangulations for manifolds admitting a genus one open book decomposition with connected binding. Our construction gives a lower bound to the complexity of the $3$--manifolds considered in the last section. The concept of layered-triangulations was introduced by Jaco and Rubinstein \cite{JR}. We refer the reader to their original paper for more details on layered-triangulations.

Suppose that $M$ is a compact $3$--manifold with boundary, and $\mathscr T$ is a triangulation of $M$. Then $\mathscr T|_{\partial M}$ is a (2-dimensional) triangulation of $\partial M$.
If $e$ is one edge of $\mathscr T|_{\partial M}$, then we can add a tetrahedron $\tilde{\Delta}$ to $\mathscr T$ along $e$. The new space $M'$ we get is still homeomorphic to $M$, but there is a new triangulation $\mathscr T'=\mathscr T\cup\Delta$, where $\Delta$ is the image of $\tilde{\Delta}$, and two faces of $\Delta$ are identified with the two triangles adjacent to $e$ in $\mathscr T|_{\partial M}$. This process is called ``layering a tetrahedron along an edge''.
As in Figure~\ref{fig:Flip}, $\mathscr T'|_{\partial M'}$ differs from $\mathscr T|_{\partial M}$ by ``flipping the diagonal $e$''.

\begin{figure}[ht]
\begin{picture}(350,130)
\put(90,0){\scalebox{0.35}{\includegraphics*{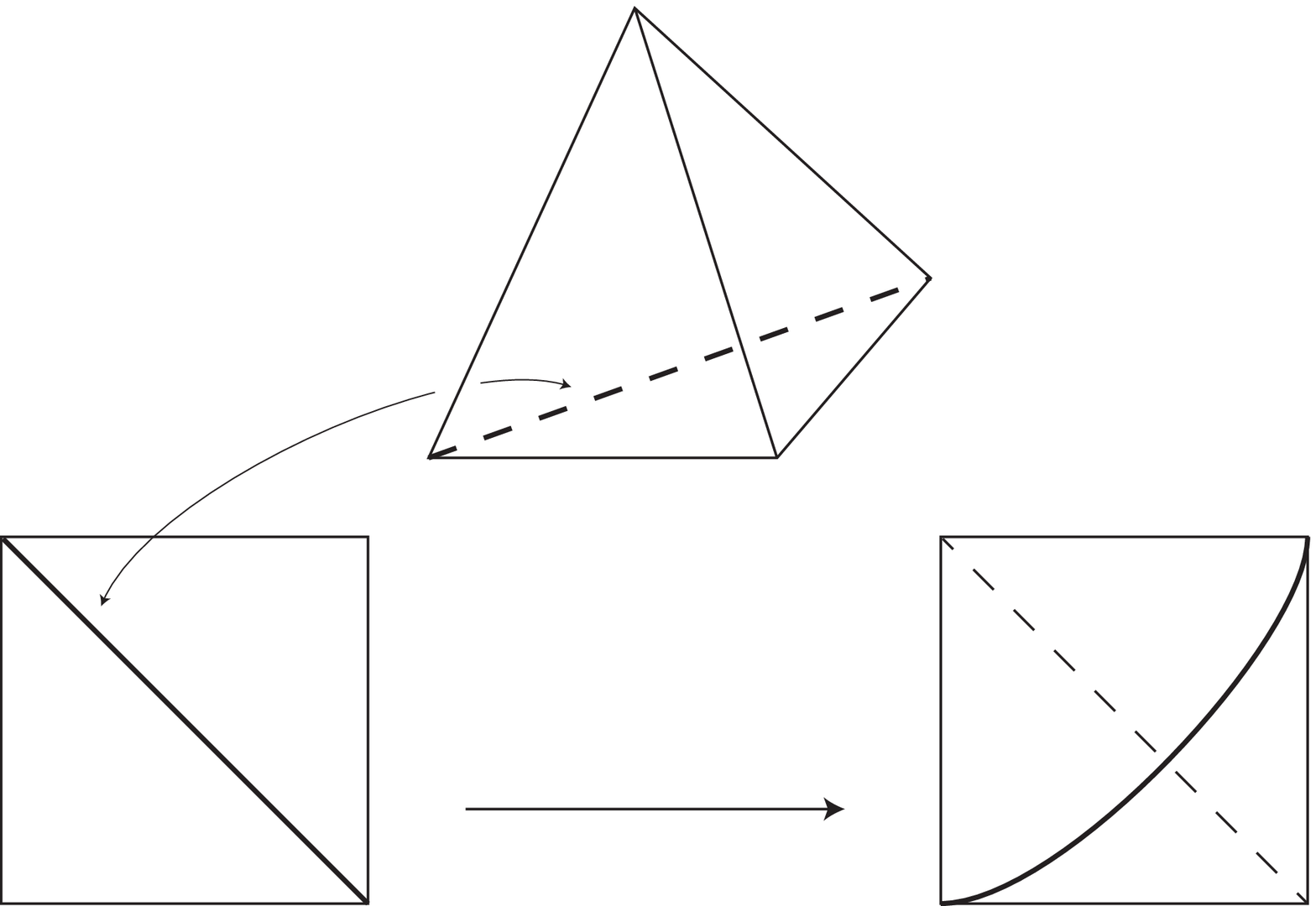}}}

\put(110,20){$e$}

\put(60,20){$\mathscr T|_{\partial M}$}

\put(150,100){$\tilde{\Delta}$}

\end{picture}
\caption{Layering a tetrahedron along an edge results in a diagonal flip in the triangulation of the boundary.}
\label{fig:Flip}
\end{figure}

The following lemma is obvious.

\begin{lem}\label{lem:UniqTr}
Let $\mathbb T_1$ be the genus $1$ compact oriented surface with one boundary component. Then $\mathbb T_1$ has exact one one-vertex triangulation $\mathscr R$ up to homeomorphism.
\end{lem}

The triangulation $\mathscr R$ has five edges, labeled as $a_1,a_2,b_1,b_2,c$, as in Figure~\ref{fig:Pentagon}. Here $c$ is the unique edge on $\partial \mathbb T_1$, and $a_1,a_2,b_1,b_2$ are called the {\it interior edges}.

Let $\mathbb S=\mathbb T_1\times[0,1]/\sim$, where the equivalence relation $\sim$ is given by $(x,t_1)\sim(x,t_2)$ for any $x\in\partial\mathbb T_1$ and $t_1,t_2\in[0,1]$. Then $\mathbb S$ is homeomorphic to a genus $2$ handlebody, and is homotopy equivalent to $\mathbb T_1$.

From now on, whenever we consider a triangulation of $\mathbb S$, this triangulation has only one vertex, and
one edge of the triangulation is on $\partial\mathbb T_1\times\{0\}$.

We start with the triangulation $\mathscr R$ of $\mathbb T_1$, regarded as a degenerate triangulation of $\mathbb S$. Suppose $\mathscr T$ is a triangulation of $\mathbb S$ we have obtained, then we can layer a tetrahedron along an interior edge of $\mathscr T|_{\mathbb T_1\times\{1\}}$ to get a new triangulation. The triangulations obtained in this way are called {\it layered-triangulations}.

Suppose that $\mathscr T$ is a layered-triangulation  of $\mathbb S$. By Lemma~\ref{lem:UniqTr}, $\mathscr T|_{\mathbb T_1\times\{0\}}$ is homeomorphic to $\mathscr T|_{\mathbb T_1\times\{1\}}$, so we can glue them together by a homeomorphism
$$f_{\mathscr T}\co \mathbb T_1\times\{1\}\to \mathbb T_1\times\{0\}$$
induced by a simplicial isomorphism.
The resulting manifold $M(\mathscr T)$ has a genus $1$ open book decomposition with connected binding, and the resulting triangulation is also called a {\it layered-triangulation}.
A well-known result of Birman--Hilden \cite{BH} implies that
$M(\mathscr T)$ is the double branched cover of $S^3$ over a closed $3$--braid, and the map $f_{\mathscr T}$ doubly covers the map on the disk with $3$ marked points corresponding to the $3$--braid.
We want to find out what the closed $3$--braid is in terms of $\mathscr T$.

\begin{figure}[ht]
\begin{picture}(350,180)
\put(30,0){\scalebox{0.35}{\includegraphics*{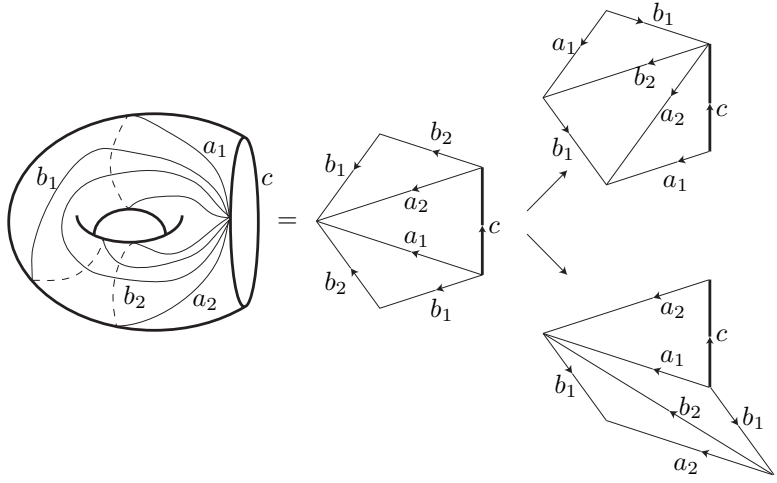}}}


\put(104,122){$a_1$}

\put(100,64){$a_2$}

 \put(41,110){$b_1$}

\put(74,65){$b_2$}

\put(126,110){$c$}

\put(180,88){$a_1$}

\put(180,101){$a_2$}

\put(190,60){$b_1$} \put(150,115){$b_1$}

\put(150,70){$b_2$} \put(190,127){$b_2$}

\put(212,92){$c$}

\put(132,94){$=$}

\put(277,134){$a_2$}

\put(267,148){$b_2$}

\put(278,109){$a_1$} \put(236,161){$a_1$}

\put(236,121){$b_1$} \put(274,172){$b_1$}

\put(298,136){$c$}

\put(277,43){$a_1$}

\put(277,61){$a_2$} \put(282,2){$a_2$}

\put(237,31){$b_1$} \put(308,19){$b_1$}

\put(284,23){$b_2$}

\put(298,51){$c$}

\end{picture}
\caption{The triangulation $\mathscr R$ and two diagonal flips of $a_1$ and $b_1$.}
\label{fig:Pentagon}
\end{figure}

\begin{lem}\label{lem:FlipMap}
Let $a_1',b_1'$ be simple closed curves in the interior of $\mathbb T_1$ which are freely isotopic to $a_1,b_1$, respectively.
Let $\tau_a,\tau_b$ be the left Dehn twists about $a'_1,b'_1$ on $\mathbb T_1$, respectively. Let $f=f_{\mathscr T}$. After layering a tetrahedron along an interior edge of $\mathscr T|_{\mathbb T_1\times\{1\}}$, we get a new triangulation $\mathscr T'$. Then the isotopy class of $f_{\mathscr T'}$ is given in the following table:
\begin{center}
\begin{tabular}{|c|c|}
\hline
\text{flipped edge} &$f_{\mathscr T'}$\\
\hline
$a_1$ &$f\tau_b^{-1}\tau_a^{-1}$\\
\hline
$a_2$ &$f\tau_a\tau_b$\\
\hline
$b_1$ &$f\tau_a^{-1}$\\
\hline
$b_2$ &$f\tau_a$\\
\hline
\end{tabular}
\end{center}
\end{lem}
\begin{proof}
For our convenience, we orient the edges $a_i,b_i$ as in Figure~\ref{fig:Pentagon}.
Homotopically, the Dehn twist $\tau_a$ fixes $a_1,a_2$, sends $b_1$ to $\bar b_2$, and sends $b_2$ to a loop freely homotopic to $b_2*a_1$;
the Dehn twist $\tau_b$ fixes $b_1$, and sends $x$ to $x*b_1$ where $x\in\{a_1,a_2,b_2\}$.

Let $a_i^{(k)},b_i^{(k)}$ be the interior edges of $\mathscr T|_{\mathbb T_1\times\{k\}}$, $i=1,2$, $k=0,1$. After layering a new tetrahedron, the interior edges of $\mathscr T'|_{\mathbb T_1\times\{1\}}$ are denoted $a_i^{(2)},b_i^{(2)}$.

We first consider layering a tetrahedra along $b_1$. The central picture in Figure~\ref{fig:Pentagon} shows the labeling of $a_i^{(1)},b_i^{(1)}$, and the bottom right picture shows the labeling of $a_i^{(2)},b_i^{(2)}$. For simplicity, we suppress the superscripts in the figure. By definition, $f_{\mathscr T}$ maps $a_i^{(1)},b_i^{(1)}$ to $a_i^{(0)},b_i^{(0)}$, and $f_{\mathscr T'}$ maps $a_i^{(2)},b_i^{(2)}$ to $a_i^{(0)},b_i^{(0)}$. From Figure~\ref{fig:Pentagon}, we see that $a_i^{(2)}$ is isotopic to $a_i^{(1)}=\tau_a(a_i^{(1)})$, $b_1^{(2)}$ is isotopic to $\bar b_2^{(1)}=\tau_a(b_1^{(1)})$, and $b_2^{(2)}$ is isotopic to $b_2^{(1)}*a_1^{(1)}=\tau_a(b_2^{(1)})$.
Hence $$f_{\mathscr T'}=f_{\mathscr T}\tau_a^{-1}$$ up to isotopy.

If we flip $b_1$, the new diagonal we create is $b_2$. Hence if we flip $b_2$, we will create $b_1$. Suppose we start with a triangulation $\mathscr T$, layer a tetrahedra along $b_1$, then layer a tetrahedra along $b_2$. The two new tetrahedra have exactly two common faces. As observed in \cite{JR}, we can crush the two tetrahedra to get a triangulation with two less tetrahedra, which is exactly the initial triangulation $\mathscr T$. In this sense, layering a tetrahedra along $b_2$ is the inverse operation of layering a tetrahedra along $b_1$.
As a result, $$f_{\mathscr T'}=f_{\mathscr T}\tau_a$$ if $\mathscr T'$ is obtained by layering a tetrahedra along $b_2$.

Now we consider
 layering a tetrahedra along $a_1$, in which case the top right picture in Figure~\ref{fig:Pentagon} shows the labeling of $a_i^{(2)},b_i^{(2)}$. We can see that
$$a_1^{(2)}\sim b_1^{(1)}=\tau_a\tau_b(a_1^{(1)}), \quad a_2^{(2)}\sim a_2^{(1)}*\bar b_2^{(1)}=\tau_a\tau_b(a_2^{(1)}),$$
$$b_1^{(2)}\sim\bar b_2^{(1)}=\tau_a\tau_b(b_1^{(1)}), \quad b_2^{(2)}\sim a_2^{(1)}=\tau_a\tau_b(b_2^{(1)}).$$
So we can conclude that
$$f_{\mathscr T'}=f_{\mathscr T}(\tau_a\tau_b)^{-1}.$$

As before, layering a tetrahedra along $a_2$ is the inverse operation of layering a tetrahedra along $a_1$. So $$f_{\mathscr T'}=f_{\mathscr T}\tau_a\tau_b$$
 if $\mathscr T'$ is obtained by  layering a tetrahedra along $a_2$.
 \end{proof}

\begin{thm}\label{thm:OBgenus1}
Suppose that $M$ is the double branched cover of $S^3$ over the closure of a $3$--braid $\sigma$, and the word length of $\sigma$ is $l$ with respect to the generator set $\{\sigma_2^{\pm1},(\sigma_2\sigma_1)^{\pm1}\}$. Then $M$ has a one-vertex triangulation with $l$ tetrahedra.
\end{thm}
\begin{proof}
By Birman--Hilden \cite{BH}, under the double branched cover from an annulus to a disk with two branched points, a Dehn twist doubly covers a half Dehn twist which corresponds to a standard generator of the braid group. Consider the double branched cover from $\mathbb T_1$ to the disk with three branched points, we can get a one-to-one correspondence between the mapping classes of $\mathbb T_1$ and the $3$--braids. (This correspondence is illustrated in \cite[Figure~9.15]{FM}.)
After labeling the three branched points appropriately, we may arrange so that $\tau_b$ corresponds to $\sigma_1$ and $\tau_a$ corresponds to $\sigma_2$. Our desired result then follows from Lemma~\ref{lem:FlipMap}.
\end{proof}

\begin{proof}[Proof of Proposition~\ref{prop:CBraid2n}]
The lower bound follows from Theorem~\ref{thm:JRT} and Proposition~\ref{prop:Braid2n}.

For the upper bound, we notice that $\sigma_2\sigma\sigma_2^{-1}$ is equal to
$$
(\sigma_2\sigma_1)\sigma_2^{-2a_1-1}(\sigma_2\sigma_1)\sigma_2^{-2a_2-1}\cdots
(\sigma_2\sigma_1)\sigma_2^{-2a_{2n-1}-1}(\sigma_2\sigma_1)\sigma_2^{-2a_{2n}-1}
$$
and apply Theorem~\ref{thm:OBgenus1}.
\end{proof}

Similarly, using Proposition~\ref{prop:Braid3}, we can prove Proposition~\ref{prop:CBraid3}.

\end{document}